\documentclass[12pt]{amsart}
\usepackage{eurosym}
\usepackage{amsmath,amssymb,amsfonts}
\usepackage[colorlinks=true,linkcolor=red,citecolor=blue]{hyperref}

\newtheorem{theorem}{Theorem}[section]
\newtheorem{proposition}[theorem]{Proposition}

\newtheorem{lemma}[theorem]{Lemma}
\newtheorem{corollary}[theorem]{Corollary}
\newtheorem*{remark}{Remark}

\newcommand\myenum[1]

\begin{document}
\title[characterizations of L-weakly compact sets using uaw-convergence]{Some characterizations of L-weakly compact sets using the unbounded absolutely weakly convergence and applications }

\author{ H. Khabaoui}
\author{ J. H'michane}
\author{K. El Fahri}

\address{Hassan Khabaoui, Universit\'{e} Moulay Ismail, Facult\'{e} des Sciences, D\'{e}partement de Math\'{e}matiques, B.P. 11201 Zitoune, Mekn\`{e}s, Morocco}
\email{khabaoui.hassan2@gmail.com}

\address{Jawad H'michane, Universit\'{e} Moulay Ismail, Facult\'{e} des Sciences, D\'{e}partement de Math\'{e}matiques, B.P. 11201 Zitoune, Mekn\`{e}s, Morocco, And, Gosaef, Faculte des Sciences de Tunis, Manar II-Tunis (Tunisia)}
\email{hm1982jad@gmail.com}

\address{Kamal El fahri, Department of Mathematics, Ibn Zohr University, Faculty of Sciences, Agadir, Morocco.}
\email{kamalelfahri@gmail.com }

\begin{abstract}
Based on the concept of unbounded absolutely weakly convergence, we give new characterizations  of L-weakly compact sets. As applications, we find some properties of order weakly compact operators. Also, a new characterizations of order continuous Banach lattices are obtained.
\end{abstract}

\maketitle

\section{Introduction}
In \cite{Zabeti}, the author introduce the concept of unbounded absolutely weakly convergence. Namely, a net $(x_{\alpha})$ is unbounded absolutely weakly convergent (uaw-convergent) to $x$ if $(|x_{\alpha} -x|\wedge u)$ converges  weakly to zero for every $u \in E^+$; we write $x_{\alpha} {\overset{uaw}{\longrightarrow}} x$. They study the relations between this concept and other concepts of convergence (weakly convergence, unbounded order and unbounded norm convergence). Also, some characterizations  of order continuous and of reflexive Banach lattices are obtained. Several recent papers investigated this concept of convergence have been announced (\cite{chen19, chen20, Zabeti19, Zabeti20}).

In this work, we study the L-weakly compact sets using the unbounded absolutely weakly convergence, in terms of uaw-convergence we give some characterizations about order (L)-Dunford-Pettis operators  and about order weakly compact operators. We give a generalisation of Theorem 4.34 \cite{AB} and also generalizations of some results  given in \cite{Zabeti, chen19} are obtained. As applications, we characterize Banach lattices under which weakly convergence implies uaw-convergence, also some new characterizations of order continuous Banach lattices are obtained.

\section{Preliminaries and notations}

Recall that  a nonempty bounded subset $A$ of a Banach lattice $E$ is said to be \textbf{L-weakly compact} if  ${\underset{n\longrightarrow +\infty}{\lim}} \|x_n\| = 0$ for every disjoint sequence $(x_n)\subset sol(A)$. A bounded subset $A$ of a Banach space $X$ is called \textbf{limited} if, for every weak* null sequence $(f_n)$ in $X'$ we have ${\underset{x\in A}{\sup}}|f_n(x)|\longrightarrow 0$. Equivalently, for every sequence $(x_n)\subset A$ and for every weak* null sequence $(f_n)$ in $X'$ we have ${\underset{n\longrightarrow +\infty}{\lim}} f_n(x_n) = 0$.  A bounded subset $A$ of a Banach space $X$ is called \textbf{Dunford-Pettis set} if, for every weak null sequence $(f_n)$ in $X'$ we have ${\underset{x\in A}{\sup}}|f_n(x)|\longrightarrow 0$. A bounded subset $A$ of $X'$ is said to be \textbf{(L)-set} if, for every weakly null sequence $(x_n)$ in $X$ we have ${\underset{f\in A}{\sup}}|f(x_n)|\longrightarrow 0$. If $E$ is a Banach lattice then a bounded subset $A$ of $E'$ is said to be \textbf{almost (L)-set} if, for every disjoint weakly null sequence $(x_n)$ in $E$ we have ${\underset{f\in A}{\sup}}|f(x_n)|\longrightarrow 0$. A net $(x_{\alpha})$ in  a Riesz space $X$ is \textbf{unbounded order convergent} (abb. uo-convergent) if, $ |x_{\alpha}-x|\wedge u$ converges  to $0$ in order (abb. $x_{\alpha}{\overset{uo}{\longrightarrow}} 0$), for each $u\in X^+$. A net $(x_{\alpha})$ in  a Banach lattice $E$ is \textbf{unbounded norm convergent} (abb. un-convergent) if, $\| |x_{\alpha}-x|\wedge u  \| \longrightarrow 0 $ (abb. $x_{\alpha}{\overset{un}{\longrightarrow}} 0$), for each $u\in E^+$. Note that the norm convergence implies the un-convergence. We can easily check, that un-convergence coincides with norm convergence for order bounded nets, in particular on a Banach lattice with order unit. The lattice operations in a Banach lattice $E$ are  weakly  sequentially continuous, if the sequence $(|x_{n}|)$  converges to $0$ in the weak topology, whenever the sequence $(x_{n})$  converges weakly  to $0$ in $E$. A vector lattice $E$ is Dedekind $\sigma$-complete if every majorized countable nonempty subset of $E$ has a supremum. A Banach lattice $E$ is order continuous, if for each net $(x_{\alpha })$ such that $x_{\alpha }\downarrow 0$ in $E$, the sequence $(x_{\alpha })$ converges to $0$ for the norm $\Vert \cdot \Vert $, where the notation $x_{\alpha }\downarrow 0$ means that the sequence $(x_{\alpha })$ is decreasing, its infimum exists and $\inf(x_{\alpha })=0$. A Banach lattice has the Schur (resp. positive Schur) property whenever  $x_n{\overset{w}{\longrightarrow}} 0$ (resp. $0\leq x_n{\overset{w}{\longrightarrow}} 0$)  implies that $\|x_n\|\longrightarrow0$.

We will use the term operator to mean a bounded linear mapping. A linear mapping between two vector lattices $E$ and $F$ is positive if $T(x)\geq 0$ in $F$ whenever $x\geq 0$ in $E$. Note that each positive linear mapping on a Banach lattice is continuous. If an operator $T:E\longrightarrow F$ between two Banach lattices is positive then, its adjoint $T^{\prime }:F^{\prime}\longrightarrow E^{\prime }$ is likewise positive, where $T^{\prime }$ is defined by $T^{\prime }(f)(x)=f(T(x))$ for each $f\in F^{\prime }$ and for each $x\in E$.

We need to recall some definitions of operators which are used along the paper.
Let $E$, $F$ be Banach lattices and $X$, $Y$ be Banach spaces.

\begin{itemize}
\item An operator $T:X\longrightarrow F$ is said to be L-weakly compact  if $T(B_X)$ is an L-weakly compact subset of $F$.

\item An operator $T:E\longrightarrow Y$ is said to be M-weakly compact  if ${\underset{n\longrightarrow +\infty}{\lim}} \|T(x_n)\| = 0$ for every disjoint bounded sequence $(x_n)\subset E$.

%\item An operator $T:X\longrightarrow Y$ is said to be weakly compact  if $T(B_X)$ is a relatively weakly compact subset of $Y$.

\item An operator $T:E\longrightarrow Y$ is said to be order weakly compact if $T([0,x])$ is a relatively weakly compact subset of $X$ for every $x$ in $E^+$.
\item An operator  $T:X\longrightarrow Y$ is said to be Dunford-Pettis whenever  ${\underset{n\longrightarrow +\infty}{\lim}} \|T(x_n)\| = 0$  for every weakly null sequence $(x_{n})$ in $X$.
 \item An operator  $T:E\longrightarrow Y$ is said to be almost Dunford-Pettis whenever  ${\underset{n\longrightarrow +\infty}{\lim}} \|T(x_n)\| = 0$  for every disjoint weakly null sequence $(x_{n})$ in $E$.

\item An operator  $T: E \longrightarrow Y$  is said to be unbounded absolutely weak Dunford–Pettis operator (add, uaw-Dunford–Pettis) if for every norm bounded sequence $x_n$ in $E$, $x_{n} \overset{uaw}\longrightarrow 0$  implies $||Tx_{n}|| \longrightarrow 0$.

\item An operator $T : X \longrightarrow F$, from a Banach space into a Riesz space is said to be order (L)-Dunford-Pettis if its adjoint $T'$ carries each order bounded subset of $F'$ into an (L)-set in $X'$.
\end{itemize}

\section{Main results}

%%%%%%%%%%%%%%%%%%%%%%%%%%%%%%%%%%%%%%%%%%%
Following [\cite{chen19}, Theorem 3.1] a non-empty bounded subset $A$ of a Banach lattice $E$ is L-weakly compact if, and only if, ${\underset{x\in A}{\sup}}|f_n(x)|\longrightarrow 0$ for every norm bounded ua-weakly null sequence $(f_n)$ in $E'$. Also, we note that every disjoint sequence of a Banach lattice is uaw-null [\cite{Zabeti}, Lemma 2]. On the other hand, it is clear that a Banach lattice $E$ is order continuous if, and only if, every order interval of $E$ is L-weakly compact.

For the next results, we need the following Lemma;
\begin{lemma}[Corollary 2.6\cite{Dodds}]\label{lemma1}\quad \\
	Let $E$ be a Banach lattice and  $(x_n)$ a sequence of $E$. The following statements are equivalents:
		\begin{enumerate}
			\item $||x_n||\longrightarrow 0$.
			\item $|x_n |\overset{w}\longrightarrow 0$ and for every norm bounded disjoint sequence $(f_n)$ of $(E')^{+}$ $f_n(x_n) \longrightarrow 0$.
		\end{enumerate}
\end{lemma}

We start by  the following characterization of L-weakly compact sets which is a generalisation of Theorem 3.3 \cite{chen19}.
\begin{proposition}\label{prop02}
Let $E$ be a Banach lattice and $A$ be a norm bounded subset of $E$. The following statements are equivalents:
	\begin{enumerate}
		\item  $A$ is L-weakly compact.
		\item For every norm bounded uaw-null sequence $(f_n)$ in $E'$ and for every  sequence $(x_n)$ in $E^+\cap sol(A)$  we have $f_n(x_n)\longrightarrow 0$.
         \item For every norm bounded uaw-null sequence $(f_n)$ in $E'$ and for every  disjoint sequence $(x_n)$ in $E^+\cap sol(A)$  we have $f_n(x_n)\longrightarrow 0$.
	\end{enumerate}
\end{proposition}

\begin{proof}
$1)\Longrightarrow 2)$ and $2)\Longrightarrow 3)$ Are obvious.

$3)\Longrightarrow 1)$ Let $(x_n)$ be a disjoint sequence in sol($A$). We have to show that ${\underset{n\longrightarrow +\infty}{\lim}} \|x_n\| = 0$. Indeed,
\begin{itemize}
\item Let $f\in E'$, then it follows from the Exercise 22 page 77 \cite{AB} that there exists a disjoint sequence $(g_n)$ of $E'$ such that $|g_n|\leq |f|$ and $g_n(|x_n|)=f(|x_n|)$, we note that $(g_n)$ is norm bounded. As $(g_n)$ is disjoint, then $(g_n)$ is a norm bounded uaw-null sequence and hence by our hypothesis ${\underset{n\longrightarrow +\infty}{\lim}}g_n(|x_n|)=0$ which implies that ${\underset{n\longrightarrow +\infty}{\lim}}f(|x_n|)=0$, that is $(|x_n|)$ is a weakly null sequence.
\item Let $(f_n)$ be a norm bounded disjoint sequence of $(E')^+$, then $(f_n)$ is a norm bounded uaw-null and hence by our hypothesis we have ${\underset{n\longrightarrow +\infty}{\lim}}f_n(x_n)=0$.
\end{itemize}
The conclusion follows from Lemma \ref{lemma1}.
\end{proof}
%%%%%%%%%%%%%%%%%%%%%%%%%%%%%%%%%%%%%%%%%%%%%%%%%%%%
%%%%%%%%%%%%%%%%%%%%%%%%%%%%%%%%%%%%%%%%%%%%%%%%%%%%
%%%%%%%%%%%%%%%%%%%%%%%%%%%%%%%%%%%%%%%%%%%%%%%%%%%%
In the following result, we give a generalization of Theorem 4.34 \cite{AB}.
\begin{theorem}\label{theo0}
If $W$ is a weakly relatively compact subset of a Banach lattice, then every uaw-null sequence $(x_n)\subset sol(W)$ converge weakly to zero.
\end{theorem}

\begin{proof}
Let $(x_n)$ be an uaw-null sequence of $sol(W)$ and let $\varepsilon>0$ and $f\in E'$. It follows from Theorem 4.37 \cite{AB}  that there exists $u\geq 0$ lying in the ideal generated by $W$ such that $|f|((|x|-u)^+)<\frac{\varepsilon}{2}$ for all $n\in \mathbb{N}$. Since $$|f|(|x_n|)=|f|\left((|x_n|-u)^{+} +|x_n|\wedge u\right)=|f|((|x_n|-u)^{+}) +|f|(|x_n|\wedge u)$$ for all $n$, and since $x_n{\overset{uaw}{\longrightarrow}}0$ then there exists $n_0\in \mathbb{N}$ such that for all $n\geq n_0$ we have $|f|(|x_n|\wedge u)<\frac{\varepsilon}{2}$.So,for all $n\geq n_0$ we have $|f(x_n)|\leq \varepsilon$. That is $x_n{\overset{w}{\longrightarrow}}0$.
\end{proof}
%As a direct consequence of Theorem \ref{theo0}, we fined the Theorem 4.34 \cite{AB}.

In the following result, we fined an other characterization of L-weakly compact sets.
\begin{proposition}\label{prop0}
For a bounded subset $A$ of a Banach lattice $E$, the following statements are equivalents:
\begin{enumerate}
		\item  $A$ is L-weakly compact.
		\item For every uaw-null sequence $(x_n)\subset sol(A)$ we have $\lim_{n\longrightarrow +\infty} \|x_n\|=0$.
	\end{enumerate}
\end{proposition}

\begin{proof}
$1)\Longrightarrow 2)$ Let $(x_n)$ be an uaw-null sequence of $sol(A)$. As $A$ is L-weakly compact, then $A$ is relatively weakly compact and hence by Theorem \ref{theo0} we have $|x_n|{\overset{w}{\longrightarrow}}0.$ On the other hand, let $(f_n)$ be a norm bounded disjoint sequence of $(E')^+$. We will show that $f_n(x_n)\longrightarrow 0$. Since $(f_n)$ is a disjoint sequence then $f_n{\overset{uaw}{\longrightarrow}}0$ and since $A$ is L-weakly compact then $f_n(x_n)\longrightarrow 0$. So, by lemma \ref{lemma1} we infer that $\lim_{n\longrightarrow +\infty} \|x_n\|=0$.

$2)\Longrightarrow 1)$ Follows from the fact that every disjoint sequence in uaw-null.
\end{proof}

As consequence of Theorem \ref{theo0} and Proposition \ref{prop0}, we fined the following result which is exactly the Corollary 3.6.8 \cite{Mey}.
\begin{corollary}\label{coro0}
For a Banach lattice $E$, the following statements are equivalents:
\begin{enumerate}
		\item  $E$ has the positive Schur property.
		\item Every relatively weakly compact subset $A$ of $E$ is L-weakly compact.
	\end{enumerate}
\end{corollary}
%%%%%%%%%%%%%%%%%%%%%%%%%%%%%%%%%%%%%%%%%%%%%%%%%%%%%%%%%%%%
%%%%%%%%%%%%%%%%%%%%%%%%%%%%%%%%%%%%%%%%%%%%%%%%%%%%%%%%%%%%
Other characterizations of L-weakly compact sets are obtained in the following result.
\begin{proposition}\label{prop21}
For an operator $T$ between tow Banach lattice $E$ and $F$, the following statements are equivalents:
\begin{enumerate}
%\item $T : E \longrightarrow X$ is an order L-weakly compact operator.
\item $T[-x,x]$ is an  L-weakly compact subset of $F$ for every $x\in E^+$.
\item For every norm bounded uaw-null sequence $(f_n)$ of $F'$, we have $|T'(f_n)| \overset{w^*}\longrightarrow 0$.
\item For every order bounded sequence $(x_n)\subset E$ and for every norm bounded uaw-null sequence $(f_n)$ of $F'$, we have ${\underset{n\longrightarrow +\infty}{\lim}} f_n(T(x_n))= 0$.
\end{enumerate}
\end{proposition}

 \begin{proof}
 $1)\Longrightarrow 2)$	Let $x\in E^{+}$. Since $T\left( \left[-x,x \right]\right)$ is an L-weakly compact set, then for every norm bounded sequence  $(f_n)$ of $F'$ with $f_n \overset{uaw}\longrightarrow 0$ we have  $\sup_{y\in T\left[ -x,x\right] }|f_n(y)|\longrightarrow 0$ and hence
$$|T'(f_n)|(x)=\sup_{z\in \left[ -x,x\right] }|T'(f_n)(z)|=\sup_{y\in T\left[ -x,x\right] }|f_n(y)|\longrightarrow 0.$$
That is $|T'(f_n)| \overset{w^*}\longrightarrow 0$.\\
$2)\Longrightarrow 3)$ and $2)\Longrightarrow 3)$  are Obvious.
\end{proof}

%\begin{remark}
%If an operator $T$ bijective positif a Banach lattice $\sigma$-Dedekind $E$ into a Banach lattice $F$ and The lattice operation of $E'$ are *weakly %sequentially continuous, then
% $T$ order ua-DP $\iff$ $T$ order weakly compact.
%\end{remark}	

As consequence of Proposition \ref{prop21} we have the following result.
\begin{corollary}
	Let $E$ be a Banach lattice. Then, the following statements are equivalents:
	\begin{enumerate}
		\item For each $x \in E^{+}$, $\left[-x,x \right]$  is an L-weakly compact set.
		\item For every norm bounded uaw-null sequence $(f_n)$ of $X'$, $|f_n |\overset{w^*}\longrightarrow 0$.
        \item $E$ is order continuous.
	\end{enumerate}
\end{corollary}

%The following result is a generalisation of the first part of Theorem 3.3 \cite{chen19}.
If in addition in the Proposition \ref{prop21}, $T$ is a positive operator then we have the following equivalences.
\begin{theorem}\label{theorem01}
Let $T:E \longrightarrow F$ be a positive operator between two Banach lattices $E$ and $F$. The following statements are equivalents:
	\begin{enumerate}
		\item  $T[-x,x]$ is an L-weakly compact subset of $F$ for every $x\in E^+$.
		\item For every limited subset $A$ of $E$, $T(A)$ is L-weakly compact subset of $F$.
       \item For every relatively compact subset $A$ of $E$, $T(A)$ is an L-weakly compact subset of $F$.
	\end{enumerate}
\end{theorem}

\begin{proof}
$1)\Longrightarrow 2)$ Let $A$ be a limited subset of $E$, $(x_n)\subset A$ and let $(f_n)$ be a norm bounded uaw-null sequence of $F'$. By Proposition 3.2 \cite{chen19} it suffice to show that ${\underset{n\longrightarrow +\infty}{\lim}}f_n(T(x_n))=0$. Since $T[-x,x]$ is an  L-weakly compact subset of $F$, then it follows from Proposition \ref{prop21} that $T'(f_n){\overset{w*}{\longrightarrow}} 0$ and hence ${\underset{n\longrightarrow +\infty}{\lim}}T'(f_n)(x_n)=0$ that is ${\underset{n\longrightarrow +\infty}{\lim}}f_n(T(x_n))=0$.

$2)\Longrightarrow 3)$ Obvious.

$3)\Longrightarrow 1)$ Let $(x_n)$ be an order bounded sequence of $E$ and $(f_n)$ be a norm bounded uaw-null sequence of $F'$. By the Proposition \ref{prop21} it suffice to show that ${\underset{n\longrightarrow +\infty}{\lim}}f_n(T(x_n))=0$. Since $(x_n)$ is order bounded, then there exists $x\in E^+$ such that $|x_n|\leq x$ for all $n\in \mathbb{N}$. On the other hand, $(f_n)$ uaw-null implies that $(|f_n|)$ is uaw-null. As the singleton $\{x\}$ is a relatively compact subset of $E$ then from our hypothesis $\{T(x)\}$ is an L-weakly compact subset of $F$. So, by Proposition 3.2 \cite{chen19} we infer
that ${\underset{n\longrightarrow +\infty}{\lim}}|f_n|(T(x))=0$. \\
From the inequality $$|f_n(T(x_n))|\leq |f_n|(T(|x_n|)) \leq |f_n|(T(x))$$
we conclude that  ${\underset{n\longrightarrow +\infty}{\lim}}f_n(T(x_n))=0$ as desired.
\end{proof}

As consequence of Theorem \ref{theorem01}, we have the following result.
\begin{corollary}
Let $E$ be a Banach lattice. Then the following statements are equivalents:
\begin{enumerate}
\item $E$ is order continuous.
\item Every limited subset $A$ of $E$ is  L-weakly compact.
\item Every relatively compact $A$ of $E$ is L-weakly compact.
       \end{enumerate}
\end{corollary}

%%%%%%%%%%%%%%%%%%%%%%%%%%%%%%%%%%%%%%%%%%%%%%%%%%%%%%%%%%%
%%%%%%%%%%%%%%%%%%%%%%%%%%%%%%%%%%%%%%%%%%%%%%%%%%%%%%%%%%%
The following result is the dual counterpart of Proposition \ref{prop21} and Theorem \ref{theorem01}.
\begin{theorem}\label{adplw}
	Let  $T:E\longrightarrow F$ be an operator between two Banach lattices $E$ and $F$. Then the following statements are equivalents:
	\begin{enumerate}
         \item $T'[-f,f]$ is an  L-weakly compact subset of $E'$ for every $f\in (F')^+$.
         \item For every norm bounded uaw-null sequence $(x_n)$ of $E$, we have $|T(x_n)| \overset{w}\longrightarrow 0$.
          \item For every order bounded sequence $(f_n)\subset F'$ and for every norm bounded uaw-null sequence $(x_n)$ of $E$, we have ${\underset{n\longrightarrow +\infty}{\lim}} f_n(T(x_n))= 0$.
		\item $T'(A)$ is an L-weakly compact subset of $E'$ for every almost (L)-set $A$ of $F'$.
		\item For every almost-Dunford-Pettis operator $S:F\longrightarrow X$ acting between the Banach lattice $F$ and an arbitrary Banach space $X$, the composed operator $S\circ T$ is an M-weakly compact operator.
	\end{enumerate}
\end{theorem}

\begin{proof}
$1)\Longleftrightarrow 2)\Longleftrightarrow 3)$ The proof is similar to that of Proposition \ref{prop21}.

$2)\Longrightarrow 4)$ By the Proposition 3.2 \cite{chen19} it suffice to prove that $T'(f_n)(x_n)\longrightarrow 0$ for every sequence  $(f_n)\subset A$ and every norm bounded uaw-null sequence $(x_n)$ of $E$. Indeed, let $(f_n)\subset A$ and  $(x_n)$ a norm bounded uaw-null sequence of $E$. It follows from our hypothesis that $|T(x_n)| \overset{w}\longrightarrow 0$ and since $A$ is an almost(L)-set then by Theorem 3.11 \cite{palermo} we have $sol(A)$ is an almost (L)-set and $|f_n(T(x_n))| \leq |f_n||T(x_n)| \leq   {\underset{g\in sol(A)}{\sup}}\{g|T(x_n)|\}\longrightarrow 0$ as desired.

$4)\Longrightarrow 5)$ Obvious.

$5)\Longrightarrow 3)$ Let $(x_n)$ be a norm bounded uaw-null sequence of $E$ and $(f_n)$ an order bounded sequence of $F'$. \\
We consider the operator
$$ \begin{array}{lrcl}
S : & F & \longrightarrow &  \ell^{\infty} \\
&x & \longmapsto &  (f_n(x)) \end{array}.$$
It is clear that the operator $S$ is will defined and is almost Dunford-Pettis. Indeed, since $(f_n)$ is order bounded then there exists $f\in (F')^+$ such that $|f_n|\leq f$ for all $n\in \mathbb{N}$. On the other hand, let $(x_n)$ be a positive weakly null sequence of $F$ then
$$\|S(x_m)\|={\underset{n}{\sup}}\{|f_n(x_m)|\}\leq f(x_n)\longrightarrow 0.$$
So, the operator $S$ is almost Dunford-Pettis and from our hypothesis $S\circ T$ is an M-weakly compact operator. By Theorem 4.5 \cite{chen19} we conclude that ${\underset{n\longrightarrow +\infty}{\lim}} \|S\circ T(x_n)\| = 0$ and hence $$|f_n(T(x_n))|\leq \|S\circ T(x_n)\|\longrightarrow 0.$$
\end{proof}

As consequence of Theorem \ref{adplw} we have the following result which generalize the result given in Theorem 7 \cite{Zabeti}.

\begin{corollary}
Let $E$ be a Banach lattice. Then the following statements are equivalents:
\begin{enumerate}
       \item $E'$ is order continuous.
         \item For every norm bounded uaw-null sequence $(x_n)$ of $E$, we have $|x_n| \overset{w}\longrightarrow 0$.
         \item For every order bounded sequence $(f_n)\subset E'$ and for every norm bounded uaw-null sequence $(x_n)$ of $E$, we have ${\underset{n\longrightarrow +\infty}{\lim}} f_n(x_n)= 0$.
	   \item Every almost (L)-set $A$ of $E'$ is L-weakly compact.
		\item Every almost Dunford-Pettis operator $S:E\longrightarrow X$ from $E$ into an arbitrary Banach space $X$ is M-weakly compact.
	\end{enumerate}
\end{corollary}

%%%%%%%%%%%%%%%%%%%%%%%%%%%%%%%%%%%%%%%%%%%%%%%%%%%%%%%%%%%%
%%%%%%%%%%%%%%%%%%%%%%%%%%%%%%%%%%%%%%%%%%%%%%%%%%%%%%%%%%%%
\begin{remark}	
%\begin{enumerate}
%\item For an operator  $T : X \longrightarrow Y$ between two Banach spaces $E$ and $F$ we have:
%               \begin{center} $T$ is a Dunford-Pettis operator if and only if $T'(B_F')$ is an (L)-set.\end{center}
It follows from Theorem 3.1 \cite{chen19} that for an operator  $T : E \longrightarrow Y$ between Banach lattice $E$ and Banach space $Y$ we have:\\
                $T$ is uaw-Dunford-Pettis operator if and only if $T'(B_Y')$ is an L-weakly compact subset of $E'$.
%\end{enumerate}
\end{remark}

In the following result, we give a sets and operators characterizations of order continuous Banach lattices.

\begin{theorem}\label{theo00}
	Let $E$ be a Banach lattice. The following statements are equivalents:
	\begin{enumerate}
		\item $E'$ is order continuous.
		\item Every (L)-set of $E'$ is an L-weakly compact subset of $E'$.
		\item Every Dunford-Pettis operators $T: E \longrightarrow X$ is uaw-Dunford-Pettis, for any arbitrary Banach space $X$.
	\end{enumerate}
\end{theorem}

\begin{proof}
$1)\Longrightarrow 2)$ Follows from Theorem 7 \cite{Zabeti} and Theorem 3.1 \cite{chen19}.

$ 2)\Longrightarrow 3)$ If $T$ is a Dunford-Pettis operator from $E$ into a Banach space $X$, then  $T'(B_X')$ is an (L)-set and hence $T'(B_X')$ is an L-weakly compact subset of $E'$, that is $T $ is uaw-Dunford-Pettis.

 $ 3)\Longrightarrow 1)$ Let $T : E \longrightarrow X$ be a Dunford-Pettis operator, by our hypothesis $T$ is uaw-Dunford-Pettis and it follows from Proposition 2.6 \cite{Zabeti19} that $T$ is weakly compact.  So, by Theorem 5.102 \cite{AB} we deduce that $E'$ is order continuous.
\end{proof}

%%%%%%%%%%%%%%%%%%%%%%%%%%%%%%%%%%%%%%%%%%%%%%%%%%%%%%%%%%
%%%%%%%%%%%%%%%%%%%%%%%%%%%%%%%%%%%%%%%%%%%%%%%%%%%%%%%%%%%%%%
In the following result we obtain a characterization of order (L)-Dunford-Pettis operators based on the uaw-convergence.
\begin{theorem}\label{theo001}
Let $T:E\longrightarrow F$ be an operator from a Banach lattice $E$ into a Riesz space $F$. The following statements are equivalents:
\begin{enumerate}
\item $T(x_{\alpha})$ is uaw-convergent for every weakly convergent net $(x_{\alpha})\subset E$.
\item $T(x_{n})$ is uaw-convergent for every weakly convergent sequence $(x_{n})\subset E$.
\item $T$ is an order (L)-Dunford-Pettis operator.
\end{enumerate}
\end{theorem}

\begin{proof}
$1)\Longrightarrow 2)$ Obvious.

$ 2)\Longrightarrow 3)$ Let $(x_n)\subset E$ be a sequence such that $ x_{n} \overset{w}\longrightarrow 0$, then the set $W=\{T(x_1),T(x_2),.... \}$ is weakly relatively compact. We have to show that  $|T(x_{n})| \overset{w}\longrightarrow 0$. For this end, fix $\epsilon >0$ and let $f \in (F')^+$, by Theorem 4.37 \cite{AB} there exists $u\in F^+ $  such that $$ f\left( (|T(x_n)|-u)^{+}\right)< \frac{\epsilon}{2} \quad \text{for all} \ n.$$
By our hypothesis, $T(x_{n}) \overset{uaw}\longrightarrow 0$. So, there exists some $n_0$ satisfying
$$ f\left( |T(x_n)|\wedge u\right) < \frac{\epsilon}{2} \quad \text{for all} \ n \geq n_0.$$
From the following inequality $$ |f\left( |T(x_n)|\right)|\leq  f\left( |T(x_n)|\right) =f\left( |T(x_n)|\wedge u\right)+f\left( (|T(x_n)|-u)^{+}\right) \leq \epsilon$$
for all $n \geq n_0$. We get that  $|T(x_{n})| \overset{w}\longrightarrow 0$ and hence it follows from Theorem 2.5 \cite{kamal} that $T$ is  an order (L)-Dunford-Pettis operator.

$3)\Longrightarrow 1)$ Suppose that $x_\alpha  \overset{w}\longrightarrow 0 $, we have to show that  $ T(x_{\alpha}) \overset{uaw}\longrightarrow 0$. Assume by way of contradiction that the claim is false. Then, there exist some $\epsilon>0$, $u\in F^+$ , $f \in (F')^+$ and a sequence $( T(x_{\alpha_ n}))$  satisfying  $f\left( |T(x_{\alpha_ n})|\wedge u\right)>\epsilon$, where $(\alpha_ n)_n$ is an increasing sequence, hence $f\left( |T(x_{\alpha_ n})|\right)>\epsilon$. Now from our hypothesis, we see that $|T(x_{\alpha_n})|  \overset{w}\longrightarrow 0 $ which implies that $f\left( |T(x_{\alpha_ n})|\right)\longrightarrow 0$ and this contradicts  the fact that $f\left( |T(x_{\alpha_ n})|\right)>\epsilon$.
\end{proof}

As consequence of Theorem \ref{theo001} we find a characterisation of Banach lattices on which weakly convergence imply uaw-convergence.
\begin{corollary}\label{coro001}
Let $E$ be a Banach lattice. The following statements are equivalents:
\begin{enumerate}
\item Every  weakly convergent net is uaw-convergent.
\item Every weakly convergent sequence is uaw-convergent.
\item The lattice operations of $E$ are weakly sequentially continuous.
\end{enumerate}
\end{corollary}

%%%%%%%%%%%%%%%%%%%%%%%%%%%%%%%%%%%%%%%%%%%%%%%%%%%%
%%%%%%%%%%%%%%%%%%%%%%%%%%%%%%%%%%%%%%%%%%%%%%%%%%%%
Note that an order (L)-Dunford-Pettis operator is not necessarily a Dunford-Pettis operator. In fact, the identity operator of the Banach lattice $c$ of all convergent sequences is order (L)-Dunford-Pettis (because the lattices operations of $c$ are weakly sequentially continuous), but is not a
Dunford-Pettis operator (because $c$ does not have the Schur property). Also, an uaw-Dunford-Pettis operators is not in general Dunford-Pettis. Indeed, if we take $E=L^2[0,1]$ and $F=c$ then it follows from Theorem 2 of \cite{weksted} that there exist two positives operators $S$ and $T$ with $S\leq T$ such that $T$ is Dunford-Pettis but $S$ is not one. As $E'$ is order continuous, then it follows from the Theorem \ref{theo00} that $T$ is uaw-Dunford-Pettis and hence $S$ is uaw-Dunford-Pettis.

As an other consequence of Theorem \ref{theo001} we have the following result.
\begin{corollary}\label{coro001}
Let $E$, $F$, $G$ be Banach lattices.
For the schema of operators $E\overset{T}\longrightarrow F\overset{S}\longrightarrow G$, if $T$ is order (L)-Dunford-Pettis  and $S$ is uaw-Dunford-Pettis then $S\circ T$ is a Dunford-Pettis operator.
\end{corollary}

%%%%%%%%%%%%%%%%%%%%%%%%%%%%%%%%%%%%%%%%%%%%%%%%%%
%%%%%%%%%%%%%%%%%%%%%%%%%%%%%%%%%%%%%%%%%%%%%%%%%%%%%

We should note from Theorem 7 \cite{Zabeti} that in a Banach lattice with order continuous bidual every Dunford-Pettis set is L-weakly compact.

As an other consequence of the Theorem \ref{theo001}, we have the following result;
\begin{corollary}\label{coro1}
If $E'$ is a Banach lattice with weak sequentially  continuous lattice operations, then every L-weakly compact set of $E$ is a Dunford-Pettis set.\\
In particular, if $E''$ is order continuous and  the lattice operations of $E'$ are weakly sequentially continuous then \textbf{L-weakly compact sets} coincide with \textbf{Dunford-Pettis sets}.
\end{corollary}

%%%%%%%%%%%%%%%%%%%%%%%%%%%%%%%%%%%%%%%%%%%%%%%%%%%%%%%%%%%%
%%%%%%%%%%%%%%%%%%%%%%%%%%%%%%%%%%%%%%%%%%%%%%%%%%%%%%%%%%%%
%%%%%%%%%%%%%%%%%%%%%%%%%%%%%%%%%%%%%%%%%%%%%%%%%%%%%%%%%%%%%
%%%%%%%%%%%%%%%%%%%%%%%%%%%%%%%%%%%%%%%%%%%%%%%%%%%%%%%%%%%%%%%%
\begin{theorem}\label{ouadp}
	Let  $T:E\longrightarrow F$ be an operator between two Banach lattices $E$ and $F$ and let $A$ be a norm bounded solid subset of $E$. Then the following statements are equivalents:
	\begin{enumerate}
		\item $T(A)$ is an  L-weakly compact subset of $F$.
		\item The subsets $ T(\left[-x,x \right])$ and $\left\lbrace T(x_n),n\in \mathbb{N} \right\rbrace $ are L-weakly compact for each $x\in   A^{+} $ and for each disjoint sequence $(x_n)$ in $A^{+}$.
		\item For every norm bounded uaw-null sequence $(f_n)$ of $F'$, $|T'(f_n) |(x)\longrightarrow 0$ for all
			$x\in   A^{+} $ and $f_n(T(x_n))\longrightarrow 0$  for each disjoint sequence $(x_n)$ in $A^{+}$.
	\end{enumerate}
\end{theorem}

\begin{proof}
It suffice to adapt the proof of Theorem 2.12 \cite{bouras} by using uaw-convergence instead of weakly convergence.
\end{proof}

As consequence of the above theorem, we have the following characterization of L-weakly compact sets.
	\begin{corollary}
	Let $E$ be a Banach lattice and let $A$ be a norm bounded solid subset of $E$. The following statements are equivalent:
		\begin{enumerate}
			\item $A$ is an  L-weakly compact set.
			\item The subsets $ \left[-x,x \right]$ and $\left\lbrace x_n,n\in \mathbb{N} \right\rbrace $ are L-weakly compact for each
			$x\in   A^{+} $ and for each disjoint sequence $(x_n)$ in $A^{+}$.
			\item For every norm bounded uaw-null sequence $(f_n)$ of $E'$, $|f_n |(x)\longrightarrow 0$ for all
			$x\in   A^{+} $ and $f_n(x_n)\longrightarrow 0$  for each disjoint sequence $(x_n)$ in $A^{+}$.
		\end{enumerate}
	\end{corollary}

	\begin{remark}
	Let $T$ be an operator from a Banach space $X$ into a Banach lattice $F$. By the equality $\sup_{y\in T(B_X)}|(f_n)(y)|=||T'(f_n)||_{X'}$ for every norm bounded uaw-null sequence $(f_n)$ in $F'$, it follows easily that $T(B_X)$ is an L-weakly compact set in $F$ if and only if $T'$ is an uaw-Dunford-Pettis operator, where $B_X$ is the closed unit ball of $X$.
	\end{remark}

%We recall from \cite{Zabeti19} that an operator $T:E\longrightarrow X$ between Banach lattice $E$ and Banach space $X$ is said to be unbounded absolutely weak Dunford–Pettis operator (abb, uaw-Dunford–Pettis) if for every norm bounded sequence $x_n$ in $E$, $x_n {\overset{uaw}{\longrightarrow}}0$ implies $\|T(x_n)\|\longrightarrow 0$.

As consequence of Theorem \ref{ouadp}, we have the following result which characterizes the adjoint of uaw-Dunford-Pettis operators between two
Banach lattices.
\begin{corollary}\label{DPunDP}
	For an operator $T$ from a Banach lattice $E$ into a Banach lattice $F$, the following statements are equivalents:
	\begin{enumerate}
		\item The adjoint $T' : F' \longrightarrow E'$ is uaw-Dunford-Pettis.
		\item $T(B_E)$ is an L-weakly compact set.
		\item $ T(\left[-x,x \right])$ is L-weakly compact for each $x\in E^+$ and $\left\lbrace T(x_n), n\in \mathbb{N}\right\rbrace $ is an L-weakly compact set for each disjoint sequence $(x_n)$ in $B^{+}_E$.
		\item $|f_n |\overset{w^*}\longrightarrow 0$ and $f_n(T(x_n))\longrightarrow 0$  for every norm bounded uaw-null sequence $(f_n)$ of $F'$ and for each disjoint sequence $(x_n)$ in $B^{+}_E$, where $B_E$ is the unit ball of $E$.
	\end{enumerate}
\end{corollary}

%%%%%%%%%%%%%%%%%%%%%%%%%%%%%%%%%%%%%%%%%%%%%%%%%%%%%%%%%%%%%%%
%%%%%%%%%%%%%%%%%%%%%%%%%%%%%%%%%%%%%%%%%%%%%%%%%%%%%%%%%%%%%%%%
%%%%%%%%%%%%%%%%%%%%%%%%%%%%%%%%%%%%%%%%%%%%%%%%%%%%%%%%%%%%%%%%%
%%%%%%%%%%%%%%%%%%%%%%%%%%%%%%%%%%%%%%%%%%%%%%%%%%%%%%%%%%%%%%%%%

Now we give a new characterization of order weakly compact operators using the unbounded absolutely weakly convergence.
\begin{proposition}
For an operator $T:E\longrightarrow X$ from a Banach lattice $E$ into a Banach space $X$, the following statements are equivalents:
\begin{enumerate}
		\item $T$ is order weakly compact.
		\item For every order bounded sequence $(x_n)\subset E^+$ such that $x_n{\overset{uaw}{\longrightarrow}}0$, we have $\lim_{n\longrightarrow +\infty} \|Tx_n\|=0$.
		\item  For every order bounded sequence $(x_n)\subset E$ such that $x_n{\overset{uaw}{\longrightarrow}}0$, we have $\lim_{n\longrightarrow +\infty} \|Tx_n\|=0$.
	\end{enumerate}
\end{proposition}
\begin{proof}
$1)\Longrightarrow 2)$ Let $(x_n)$ be an order bounded sequence of $E^+$ with $x_n{\overset{uaw}{\longrightarrow}}0$, then there exists $x\in E^+$ such that $0\leq x_n\leq x$ for all $n$. On the other hand $x_n{\overset{uaw}{\longrightarrow}}0$ implies that $x_n=x_n\wedge x{\overset{w}{\longrightarrow}}0$ and hence it follows from Corollary 3.4.9 \cite{Mey} that $\lim_{n\longrightarrow +\infty} \|Tx_n\|=0$.

$2)\Longrightarrow 3)$ Follows from the fact that $x_{n}{\overset{uaw}{\longrightarrow}}0$ implies that $x^{+}_{n}{\overset{uaw}{\longrightarrow}}0$ and $x^{-}_n{\overset{uaw}{\longrightarrow}}0$ and from the inequality $\|T(x_n)\|\leq \|T(x^{+}_{n})\|+\|T(x^{-}_{n})\|$.

$3)\Longrightarrow 1)$ Follows from the fact that every disjoint norm bounded sequence is uaw-null.
\end{proof}

As consequence, we have another characterization of order continuous Banach lattices.
\begin{corollary}\label{coro5}
	For a Banach lattice $E$, the following statements are equivalents:
	\begin{enumerate}
		\item $E$ is order continuous.
		\item For every order bounded sequence $(x_n)\subset E^+$ such that $x_n{\overset{uaw}{\longrightarrow}}0$, we have $\lim_{n\longrightarrow +\infty} \|x_n\|=0$.
		\item  For every order bounded sequence $(x_n)\subset E$ such that $x_n{\overset{uaw}{\longrightarrow}}0$, we have $\lim_{n\longrightarrow +\infty} \|x_n\|=0$.
	\end{enumerate}	
\end{corollary}

We should note that a limited subset of a Banach space is a Dunford-Pettis set, but the converse is not holds in general. By using Theorem 4.42 \cite{AB}, we can generalize the Theorem 2.6 \cite{bouras} as follows.
\begin{theorem}\label{theo1}
Let $E$ be a  Dedekind $\sigma$-complete Banach lattice and let $A$ be an order bounded set of
$E$. If $A$ is $|\sigma|\left( E,E'\right)$ -totally bounded, then $A$ is a limited set.
\end{theorem}

\begin{proof}
We use the some proof of Theorem 2.6 \cite{bouras} We just applied Theorem 4.42 \cite{AB} instead of Theorem 4.37 \cite{AB} which is used to derive the result given in Theorem 2.6 \cite{bouras}.
\end{proof}
%  \begin{proof}
%	Let $(f_n)$ be a weak* null sequence in $E'$ and let $x \in E^{+}$ such that $|y| \leq x$
%	for every $y \in A$. Fix  $\epsilon > 0$, by  Theorem 4.42 \cite{AB} there exists $g  \in  (E')^{+}$ such that
%	$(|f_n| - g)^{+}(x) <\epsilon$
%	 for each n.
%	Since $A$ is  $|\sigma|\left( E,E'\right)$ -totally bounded, there exists a finite set $\left\lbrace x_1,x_2,.....,x_k\right\rbrace    \subset A$,
%	such that for each $z \in A$, we have $g(|z - x_i|) < \frac{\epsilon}{4}$
%	 for at least one $1 \leq i \leq k $.
%	Since  $f_n \overset{w^{*}}\longrightarrow 0$, there exists $N$ with $|f_n(x_i)|<\frac{\epsilon}{4}$
%	 for each $i = 1,..... , k$ and
%	all $n \geq N$.
%	Now, let $z \in A $, choose $1 \leq i \leq k$ with $g(|z-x_i|)<\frac{\epsilon}{4}$
%	 and note that $|z-x_i|\leq 2x$
%	holds. In particular, for $n \geq N $, we have
%	$$ |f_n(z)|\leq f_n(|z - x_i|)+|f_n(x_i)|\leq (|f_n| - g)^{+}(|z-x_i|)+ g(|z-x_i|)+ |f_n(x_i)| \leq \frac{\epsilon}{4}\leq \epsilon,$$
%	the above inequality implies that $\sup_{x\in A}|f_n(z)|\longrightarrow 0$ and then $A$ is a limited set.
%\end{proof}

As consequence,

\begin{corollary}\label{coro3}
Let $E$ be a Dedekind $\sigma$-complete Banach lattice. Then, for every order bounded disjoint sequence $(x_n) $, the subset $\left\lbrace  x_n, n \in \mathbb{N} \right\rbrace $ is a limited set.
\end{corollary}

\begin{proof}
	By  Corollary 3.43 \cite{AB}, the subset $\left\lbrace  x_n, n \in \mathbb{N} \right\rbrace $ is $|\sigma|\left( E,E'\right)$ -totally bounded.	
\end{proof}

We recall that an operator $T:X\longrightarrow Y$ between two Banach space is said to be bounded below if $T$ is one-to-one and has closed range. On the other hand, an operator $T:X\longrightarrow Y$ is onto (respectively, bounded below) if and only if $T'$ is bounded below (respectively, onto).

\begin{lemma}\label{lemma00}
Let $E$ and $F$ be two Banach lattices. If $T:E\longrightarrow F$ is a positive operator which is one-to-one and onto, then $T'$ is a lattice homomorphism.
\end{lemma}

\begin{proof}
Let $f \in F'$, $g \in F'$ and $x \in E^+$, we have
\begin{eqnarray*}
% \nonumber % Remove numbering (before each equation)
  [T'(f)\vee T'(g)](x) &=&\sup\left\lbrace T'(f)(y) +T'(g)(z), y+z=x \right\rbrace \\
    &=& \sup\left\lbrace f(Ty) +g(Tz): y, z\in E^+ \text{ and } y+z=x \right\rbrace\\
   &=& \sup\left\lbrace f(Ty) +g(Tz): y, z\in E^+ \text{ and }  Ty+Tz=Tx \right\rbrace \\
    &=&(f\vee g)(Tx) = T'(f\vee g)(x).
\end{eqnarray*}
\end{proof}

\begin{theorem}
Let $E$, $F$ be two Banach lattices with $E$ Dedekind $\sigma$-complete and let $T:E\longrightarrow F$ be a positif and  one-to-one and onto operator. The following statements are equivalents:
	\begin{enumerate}
		\item $|T'(f_n)| \overset{w*}\longrightarrow 0$ for every norm bounded uaw-null sequence $(f_n)$ of $F'$.
		\item For every limited subset $A$ of $E$, we have $T(A)$ is an L-weakly compact subset of $F$.
		\item  $T$ is order weakly compact.
	\end{enumerate}
\end{theorem}
\begin{proof}
$1)\Longrightarrow 2)$ Let $(f_n)$ be a norm bounded uaw-null sequence in $F'$, then $T'(f_n) \overset{w*}\longrightarrow 0$. As $A$ is a limited subset of $E$, then we have ${\underset{x\in A}{\sup}}|T'(f_n)(x)|\longrightarrow 0$ and therefore  ${\underset{x\in A}{\sup}}|(f_n(Tx))|\longrightarrow 0$, that is ${\underset{z\in T(A)}{\sup}}|(f_n(z))|\longrightarrow 0$ which shows that $T(A)$ is an L-weakly compact subset of $F$.

$2)\Longrightarrow 3)$ Let $(x_n)$ be an order bounded disjoint sequence of $E$ and $(f_n)$ be a disjoint norm bounded sequence  of $F'$.
Then by Corollary \ref{coro3}  the subset $\left\lbrace  x_n, n \in \mathbb{N} \right\rbrace $ is limited and by our hypothesis the subset $\left\lbrace T( x_n), n \in \mathbb{N} \right\rbrace $ is L-weakly compact. As $(f_n)$ is a disjoint norm bounded sequence of $F'$, then $f_n \overset{uaw}\longrightarrow 0$ and hence $f_n(T(x_n)) \longrightarrow 0$. On the other hand, since $(x_n)$ be an order bounded disjoint sequence of $E$ then $|x_n |\overset{w}\longrightarrow 0$ which implies that $|Tx_n |\overset{w}\longrightarrow 0$  and hence  by Lemma \ref{lemma1} we have $||T(x_n)||\longrightarrow 0$, which shows that $T$ is an order weakly compact operator.

$3)\Longrightarrow 1)$ Since $T:E\longrightarrow F$ is order weakly compact, then by Theorem 5.58 \cite{AB} $T$ admits a factorization $T=S\circ Q$ through an order continuous Banach lattice $G$ such that:
\begin{enumerate}
	\item  $Q$ is a lattice homomorphism.
	\item $S$ is positive if $T$ is also positive.
	\end{enumerate}

%Note that $S$ is bijective. On the other hand, the adjoint operator $S'$ of $ S$ is onto.
 Since $T$ is one-to-one and onto then $S$ is also one-to-one and onto, so $S$ is one-to-one and bounded below and hence $S'$ is onto.
 %On the other hand, the adjoint operator $S'$ of $ S$ is onto. Indeed, for every $g \in G'^{+}$ it suffice to consider the will defined functional $f : F\longrightarrow \mathbb{R} $ defined by $f(Sx)=g(x)$ for every $x\in G^+$ (because $S$ is onto).
Moreover, by Lemma \ref{lemma00} we see that $S'$  is a lattice homomorphism.

Now, we are in position to show the assertion (1). Let $(f_n)$ be a norm bounded uaw-null sequence in $F'$ and $g \in G'^{+}$, then there exits $f \in (F')^+$ such that $S'(f)=g$. We have that $S'(f_n)\overset{uaw}\longrightarrow 0$, indeed: from the equality
                         $$S'(f_n)\wedge g =S'(f_n)\wedge S'(f)= S'(f_n \wedge f)$$
and from the fact that  $f_n\wedge f\overset{w}\longrightarrow 0$ we infer that $S'(f_n)\wedge g=S'(f_n\wedge f)\overset{w}\longrightarrow 0$, that is to say $S'(f_n)\overset{uaw}\longrightarrow 0$. Since $G$ is order continuous, it follows from Proposition 5 \cite{Zabeti} that $|S'(f_n)|\overset{w^*}\longrightarrow 0$, that is to say $|S'(f_n)|(x)\longrightarrow 0$ for every $x \in G$. Finally, it follows from the equality
\begin{eqnarray*}
	% \nonumber % Remove numbering (before each equation)
|T'(f_n)|(x) &=&\sup\left\lbrace |T'(f_n)(z)| , |z|\leq x \right\rbrace \\
	&=& \sup\left\lbrace |f_n(SQz)| ,  |Qz|\leq Qx \right\rbrace\\
	&\leq & \sup\left\lbrace|S'f_n(y)| ,  |y|\leq Qx \right\rbrace \\
		&=&|S'(f_n)|(Qx) \longrightarrow 0.
\end{eqnarray*}
for all $x\in E^+$, that $|T'(f_n)|\overset{w^{*}}\longrightarrow 0$ as desired.
\end{proof}

%As consequence, we have the following characterization of order continuous Banach lattice.
%\begin{corollary}\label{coro3}
%	Let $E$ be a $\sigma$-Dedekind  Banach lattice. The following statements are equivalents:
%	\begin{enumerate}
%		\item $E$ is order continuous.
%		\item Every limited subset of $E$ is L-weakly compact.
%	\end{enumerate}	
%	\end{corollary}

\begin{proposition}\label{unruaw}
Let $E$, $F$ be two Banach lattices and let $T:E\longrightarrow F$ be a lattice isomorphism.\\
If $(T(x_n))$ is an un-null sequence for every sequence $(x_n)\subset E$, then $(x_n)$ is an uaw-null sequence.
\end{proposition}
\begin{proof}
Let $f\in E'$, $(x_n)\subset E$ such that $(T(x_n))$ is an un-null sequence and $u\in E^+$. Since $T$ is a lattice isomorphism, then $T$ is bounded below and hence $T'$ is onto. So, there exists $g\in F'$ such that $T'(g)=f$. From the equality
\begin{eqnarray*}
% \nonumber % Remove numbering (before each equation)
  f(|x_n|\wedge u) &=& T'(g)(|x_n|\wedge u) \\
     &=& g(T(|x_n|\wedge u)) \\
     &=& g(T(|x_n|)\wedge T(u)) \\
      &=& g(|T(x_n)|\wedge T(u)) \\
    &\leq& \|g\|\cdot\||T(x_n)|\wedge T(u)\|
\end{eqnarray*}
we infer that $x_n\overset{uaw}\longrightarrow 0$.
\end{proof}

\begin{theorem}\label{uawun}
Let $E$, $F$ be two Banach lattices and let $T:E\longrightarrow F$ be an onto lattice homomorphism. Then the following statements are equivalents:
	\begin{enumerate}
		\item $(T(x_n))$ is un-null for every uaw-null sequence $(x_n)\subset E$.
		\item $(T(x_n))$ is un-null for every disjoint sequence $(x_n)\subset E$.
		\item $T$ is order weakly compact.
	\end{enumerate}
\end{theorem}

\begin{proof}
$1)\Longrightarrow 2)$  Obvious.

$2)\Longrightarrow 3)$ Let $(x_n)$ be an order bounded disjoint sequence of $E$. Since $(x_n)$ is a disjoint sequence, then by our hypothesis  $(T(x_n))$ is un-null and since $T$ is order bounded then $(T(x_n))$ is an order bounded sequence and hence  ${\underset{n\longrightarrow +\infty}{\lim}} \|T(x_n)\| = 0$, that is $T$ is order weakly compact.

$3)\Longrightarrow 1)$ Let $(x_n)$ be a norm bounded uaw-null sequence of $E$ and let $v\in F^+$. Since $T$ is an onto lattice homomorphism, then there exists $u\in E^+$ such that $T(u)=v$ and
\begin{eqnarray*}
% \nonumber % Remove numbering (before each equation)
  \||T(x_n)|\wedge v\| &=&  \||T(x_n)|\wedge T(u)\|\\
   &\leq& \|T(|x_n|\wedge u)\|.
\end{eqnarray*}
On the other hand, $x_n{\overset{uaw}{\longrightarrow}} 0$ implies that $|x_n|\wedge u{\overset{w}{\longrightarrow}} 0$ and hence $(|x_n|\wedge u)$ is a positive order bounded sequence of $E$. As $T$ is order weakly compact, it follows from Corollary 3.4.9 \cite{Mey} that ${\underset{n\rightarrow +\infty}{\lim}} \|T(|x_n|\wedge u)\|= 0$ and hence ${\underset{n\rightarrow +\infty}{\lim}} \||T(x_n)|\wedge v\|= 0$, that is $T(x_n){\overset{un}{\longrightarrow}} 0$.

%Now, let $\varepsilon>0$. As $T$ is order weakly compact, it follows from Theorem 5.57 \cite{AB} that there exist $f\in E^+$ and $\delta>0$ such that $|y|\leq u$ and $f(|y|)<\delta$ implies that $\|T(y)\|<\varepsilon$. On the other hand, $x_n{\overset{uaw}{\longrightarrow}} 0$ implies that $|x_n|\wedge u{\overset{w}{\longrightarrow}} 0$, that is there exists $n_0\in \mathbb{N}$ such that for all $n \geq n_0$ we have $|f(|x_n|\wedge u)|< \delta$. So, we have $||x_n|\wedge u|\leq u$ and $|f(|x_n|\wedge u)|< \delta$ and hence $\|T(|x_n|\wedge u)\|< \varepsilon$ which implies that $\||T(x_n)|\wedge v)\|< \varepsilon$ for all $n \geq n_0$. That is  $T(x_n){\overset{un}{\longrightarrow}} 0$.
\end{proof}

As consequence of Proposition \ref{unruaw} and Theorem \ref{uawun}, we find the following characterization of order continuous norm.
	\begin{corollary}
	Let $E$ be a Banach lattice. The following statements are equivalents:
		\begin{enumerate}
			\item $(x_n)$ is uaw-null if and only if $(x_n)$ is un-null.
			\item Every disjoint sequence $(x_n)\subset E$ is un-null.
			\item $E$ is order continuous.
		\end{enumerate}
	\end{corollary}
%%%%%%%%%%%%%%%%%%%%%%%%%%%%%%%%%%%%%%%%%%%ua-L set%%%%%%%%%%%%%%%%%%%%%%%%%%%%%
%%%%%%%%%%%%%%%%%%%%%%%%%%%%%%%%%%%%%%%%%%%%%%%%%%%%%%%%%%%%%%%%%%%%%%%%%%%%%%%%%%%%%%
%%%%%%%%%%%%%%%%%%%%%%%%%%%%%%%%%%%%%%%%%%%ua-L set%%%%%%%%%%%%%%%%%%%%%%%%%%%%%


\begin{thebibliography}{99}


\bibitem{AB} C.D. Aliprantis and O. Burkinshaw, Positive Operators. Reprint of the 1985 original. Springer, Berlin, 2006.
\bibitem{bouras} K. Bouras, Dunford-Pettis sets in Banach lattices. Acta Math. Univ. Comenianae, Vol. LXXXI, 2(2012), pp. 185-196.
\bibitem{chen19} Zhangjun Wang, Zili Chen and Jinxi Chen,  On the unbounded convergence in Banach lattices and related operators. arXiv:1903.02168v11 [math.FA] 6 Mar 2020.
\bibitem{chen20} Zhangjun Wang, Zili Chen and Jinxi Chen, Weak L-weakly compact and weak M-weakly compact operators on Banach lattices. arXiv:1903.04854v4 [math.FA] 14 Mar 2020.

\bibitem{Dodds} Dodds, P.G. and Fremlin, D.H., Compact operators on Banach lattices, Israel J. Math. 34 (1979), 287-320.

\bibitem{palermo} K. El Fahri, F.Z. Oughajji, On the class of almost order (L) sets and applications. Rendiconti del Circolo Matematico di Palermo Series 2 https://doi.org/10.1007/s12215-020-00493-7.

\bibitem{kamal} K. El Fahri, N. Machrafi, J. H'michane and A. Elbour, Application of (L) sets to some classes of operators. Mathematica Bohemica. 141 (2016), No. 3, 327–338.

\bibitem{weksted} A. W.Wickstead. Converses for the Dodds-Fremlin and Kalton-Saab theorems. {\em Math. Proc. Cambridge Philos Soc}. 120 (1996) 175-179.

\bibitem{Zabeti} Omid Zabeti, Unbounded absolute weak convergence in Banach lattices. Positivity, DOI 10.1007/s11117-017-0524-7.
\bibitem{Zabeti19} Nazife ERKURŞUN ÖZCAN, Niyazi Anıl GEZER and Omid ZABETI, Unbounded absolutely weak Dunford–Pettis operators. Turkish Journal of Mathematics, (2019) 43: 2731 – 2740, Doi:10.3906/mat-1904-27.
\bibitem{Zabeti20} Omid Zabeti, Unbounded continuous operators in Banach lattices. arXiv:1911.10015v3 [math.FA] 24 Mar 2020.
%\bibitem{HUI} HUI LI, Niyazi Anil GEZER2, ZILI CHEN, SOME RESULTS ON Unbounded absolutely weak Dunford–Pettis operators.
\bibitem{Mey} P. Meyer-Nieberg, Banach lattices, Universitext. Springer-Verlag, Berlin, 1991.

\end{thebibliography}
\end{document}